\documentclass[a4]{article}
\usepackage{graphicx}
\usepackage{color}
\usepackage{amsmath}
\usepackage{amssymb}
\usepackage{amscd}
\usepackage{framed}
\usepackage{bbm}

\usepackage{fancyhdr,a4wide}
\usepackage{amsthm}
\usepackage{xcolor}
\usepackage{mathtools}
\usepackage{caption}

\usepackage{hyperref}

\newcommand{\R}{\mathbb{R}}
\newcommand{\inr}[1]{\left\langle #1 \right\rangle}

\newcommand{\E}{\mathbb{E}}

\newcommand{\FF}{{\mathcal{F}}}
\newcommand{\GG}{{\mathcal{G}}}

\newcommand{\eps}{\varepsilon}
\newcommand{\conv}{\mathop{\rm conv}}

\newcommand{\triple}[1]{\left\vert\!\left\vert\!\left\vert {#1} \right\vert\!\right\vert\!\right\vert}

\newtheorem{Theorem}{Theorem}[section]
\newtheorem{Lemma}[Theorem]{Lemma}

\newtheorem{Definition}[Theorem]{Definition}

\newtheorem{Proposition}[Theorem]{Proposition}

\newtheorem{Corollary}[Theorem]{Corollary}
\newtheorem{Remark}[Theorem]{Remark}

\newtheorem{Assumption}[Theorem]{Assumption}
\newtheorem{Question}[Theorem]{Question}

\numberwithin{equation}{section}

\def \proof {\noindent {\bf Proof.}\ \ }

\def \endproof
{{\mbox{}\nolinebreak\hfill\rule{2mm}{2mm}\par\medbreak}}
\def\IND{\mathbbm{1}}

\newcommand{\iid}{i.i.d.\ }

\newcommand{\norm}{\|\cdot\|}

\newcommand{\PROB}{\mathbb{P}}

\newcommand{\LL}{{D}}
\def\IND{\mathbbm{1}}

\makeatletter
\newcommand{\footremember}[2]{%
\footnote{#2}
\newcounter{#1}
\setcounter{#1}{\value{footnote}}%
}
\newcommand{\footrecall}[1]{%
\footnotemark[\value{#1}]%
}
\makeatother

\title{On the geometry of polytopes generated by heavy-tailed random vectors}
\date{\today}
\author{Olivier Gu\'edon\footremember{Upem}{Universit\'e Paris-Est, Laboratoire d'Analyse et de Math\'ematiques Appliqu\'ees (UMR 8050), UPEM, UPEC, CNRS, F-77454, Marne-la-Vall\'ee, France (\href{mailto:olivier.guedon@u-pem.fr}{olivier.guedon@u-pem.fr})} \and Felix Krahmer\footremember{TUM}{Department of Mathematics, Technical University of Munich, 85748 Garching bei M\"unchen, Germany (\href{mailto:felix.krahmer@tum.de}{felix.krahmer@tum.de}, \href{mailto:c.kuemmerle@tum.de}{c.kuemmerle@tum.de})} \and Christian K\"ummerle\footrecall{TUM} \and Shahar Mendelson \footremember{Sor}{LPSM, Sorbonne University, Paris, France and Mathematical Sciences Institute, The Australian National University, Canberra, Australia. (\href{shahar.mendelson@upmc.fr}{shahar.mendelson@anu.edu.au})} \and
Holger Rauhut\footremember{RWTH}{Chair for Mathematics of Information Processing, RWTH Aachen University, 52056 Aachen, Germany (\href{rauhut@mathc.rwth-aachen.de}{rauhut@mathc.rwth-aachen.de})}}

\begin{document}

\maketitle

\begin{abstract}
We study the geometry of centrally-symmetric random polytopes, generated by $N$ independent copies of a random vector $X$ taking values in $\R^n$. We show that under minimal assumptions on $X$, for $N \gtrsim n$ and with high probability, the polytope contains a deterministic set that is naturally associated with the random vector---namely, the polar of a certain floating body. This  solves the long-standing question on whether such a random polytope contains a canonical body.
Moreover, by identifying the floating bodies associated with various random vectors we recover the estimates that have been obtained previously, and thanks to the minimal assumptions on $X$ we derive estimates in cases that had been out of reach, involving random polytopes generated by heavy-tailed random vectors (e.g., when $X$ is $q$-stable or when $X$ has an unconditional structure). Finally, the structural results are used for the study of a fundamental question in compressive sensing---noise blind sparse recovery.

\end{abstract}

\section{Introduction} \label{sec:introduction}
Let $X$ be a symmetric random vector in $\R^n$ and let $X_1,\ldots,X_N$ be independent copies of $X$. The goal of this article is to study the geometry of the random polytope
${\rm absconv}(X_1, \ldots,X_N),$
that is, the convex hull of the points $\pm X_1, \ldots, \pm X_N$.
Various aspects of the geometry of such random polytopes have been the subject of extensive study for many years. As a starting point, let us formulate two notable results in the direction we are interested in, and to that end, denote by $B_p^n$ the unit ball in $\ell_p^n$.
\begin{Theorem} \label{thm:Gluskin} \cite{G}
Let $X$ be the standard Gaussian random vector in $\R^n$, set $0<\alpha<1$ and consider $N \geq c_0(\alpha) n$. Then
\begin{equation} \label{eq:gaussian-est}
c_1(\alpha) \sqrt{\log(eN/n)} B_2^n \subset {\rm absconv}(X_1,\ldots,X_N)
\end{equation}
with probability at least $1-2\exp(-c_2N^{1-\alpha} n^{\alpha})$. Here $c_0$ and $c_1$ are constants that depend on $\alpha$ and $c_2$ is an absolute constant.
\end{Theorem}

Theorem \ref{thm:Gluskin} can be extended beyond the Gaussian case, to random polytopes generated by a random vector $X=(\xi_1,\ldots,\xi_n)$ where the $\xi_i$'s are independent copies of a mean-zero, variance $1$ random variable $\xi$ that is $L$-subgaussian\footnote{Recall that a centered random variable is $L$-subgaussian if for every $p \geq 2$,
$\left(\E |\xi|^p\right)^{1/p} \leq L \sqrt{p}$.}. This class of random vectors includes, in particular, the Rademacher vector $X$, which is uniformly distributed in $\{-1,1\}^n$.

A version of Theorem \ref{thm:Gluskin} for the Rademacher vector was established in \cite{GH} (with a slightly suboptimal dependence of $N$ on the dimension $n$). 
The optimal estimate when $\xi$ is an arbitrary subgaussian random variable is the following special
case of a result in  \cite{LPRT}.

\begin{Theorem} \label{thm:LPRT} \cite{LPRT}
Let $\xi$ be a mean-zero random variable that has variance $1$ and is $L$-subgaussian, and set $X=(\xi_i)_{i=1}^n$ as above. Let $0<\alpha<1$ and set $N \geq c_0(\alpha,L)n$. Then there exists an absolute constant $c_1$ such that with probability
at least $1-2\exp(-c_1N^{1-\alpha}n^{\alpha})$
\begin{equation} \label{eq:subgaussian-est}
c_2(\alpha,L)\bigl(B_\infty^n \cap \sqrt{\log(eN/n)} B_2^n\bigr) \subset {\rm absconv}(X_1,\ldots,X_N).
\end{equation}
\end{Theorem}

In both cases, the typical random polytope ${\rm absconv}(X_1,\ldots,X_N)$ contains a large regular convex body: a multiple of the Euclidean unit ball when $X$ is the standard Gaussian random vector, and an intersection body of two $\ell_p$ balls when $X$ is $L$-subgaussian and has \iid coordinates. As we explain in what follows, the fact that the bodies that are contained in ${\rm absconv}(X_1,\ldots,X_N)$ are different in these two examples is not a coincidence. Rather, it reflects the fact that a subgaussian random vector may in general generates a different geometry than the Gaussian one.

\medskip

Motivated by these two facts, we study the following questions:
\begin{Question} \label{Qu:main}
\begin{description}
\item{(1)} Is it possible to find a set $K$ that is naturally associated with $X$ and is contained in ${\rm absconv}(X_1,\ldots,X_N)$ with high probability?
\item{(2)} If the answer to $(1)$ is yes, when does $K$ contain large (intersections of) $\ell_p$ balls as, for example, in Theorems \ref{thm:Gluskin} and \ref{thm:LPRT}?
\end{description}
\end{Question}
Both Theorem~\ref{thm:Gluskin} and Theorem~\ref{thm:LPRT}, as well as the numerous other results in this direction, can be explained by a general principle stated in our main result, Theorem~\ref{thm:main1}, that answers part (1) of Question \ref{Qu:main}. 
The geometric features of $X$ that are significant in this context 
are reflected by the natural \emph{floating bodies} associated with $X$.
Part (2) of Question~\ref{Qu:main} will be answered in Section~\ref{sec:K-bodies}
by identifying those floating bodies for a variety of choices of $X$---thus recovering, and at times improving, previously known results, as well as establishing new estimates in cases that were out of reach before.
\begin{Definition}
Let $X$ be a symmetric random vector in $\R^n$. For $p \geq 1$, we define the  associated floating body
$$
K_{p}(X) := \left\{ t \in \R^n : \PROB( \inr{X,t} \geq 1) \leq \exp(-p) \right\}.
$$
\end{Definition}
The notion of floating bodies plays a crucial role in the study of approximation of convex bodies by polytopes, see, e.g., \cite{SW1, MR1, Barany}, where $X$ is distributed according to the uniform probability measure on the given convex body. It is known how to identify the floating bodies associated to Gaussian or Rademacher random vectors, see below. 

In order to continue we require the following notation.
Given sets $A,B \subset \R^n$, $A \sim B$ denotes that there are absolute constants $c_1$ and $c_2$ such that $c_1 A \subset B \subset c_2B$. We write $A \sim_{\kappa} B$ if the constants $c_1$ and $c_2$ depend on the parameter $\kappa$. Identifying each $t \in \R^n$ with the linear functional $\inr{\cdot,t}$, we define, for $p > 0$, the $L_p(X)$ (quasi-) norm of $t \in \R^n$ to be $\|\inr{X,t}\|_{L_p} = \left(\E | \inr{X,t} |^p \right)^{1/p}$, and denote its unit ball by
$$
B(L_p(X)) := \{ t \in \R^n : \|\inr{X,t}\|_{L_p} \leq 1\}.
$$
For $1 \leq q < \infty$ and $t = (t_1, \ldots, t_n)$, let
\[
\|t\|_q = \left(\sum_{i=1}^n |t_i|^q \right)^{1/q} \hbox{ \ and \ }
\|t\|_{\infty} = \max_{i=1, \ldots, n} |t_i|.
\]
For $1 \le q \le \infty,$
let $B_q^n=\{t \in \R^n : \|t\|_q \leq 1\}$
 be the unit ball of the normed space $\ell_q^n$, and set $q^\prime$ to be the conjugate index of $q$; that is, $\frac{1}{q}+\frac{1}{q^\prime} =1$. Finally, for $T \subset \R^n$ let
$$
T^\circ = \{x \in \R^n \ : \inr{t,x} \leq 1 \ \ {\rm for \ every \ } t \in T\};
$$
the set $T^\circ$ is the \emph{polar body} of $T$, which is 
a convex, centrally symmetric subset of $\R^n$ if $T$ is centrally symmetric.

With this notation in place, consider the following examples:
\begin{description}
\item{$\bullet$} Let $X=G$ be the standard Gaussian random vector in $\R^n$. Then for every $p \geq 1$,
    \begin{equation}\label{Gaussian-floating}
    K_{p}(G) \sim  (1/\sqrt{p}) B_2^n,
    \end{equation}
    which can be shown by a direct calculation using the rotation invariance of $G$.
    Thus, the polar body of $K_{p}(G)$ satisfies $(K_{p}(G))^{\circ} \sim \sqrt{p} \, B_2^n$.
\item{$\bullet$} Let $X={\cal E}$ be the Rademacher random vector in $\R^n$ (i.e., $X$ is uniformly distributed in $\{-1,1\}^n$). Results in \cite{Ms} imply that
    $$
    K_{p}({\cal E}) \sim {\rm conv}(B_1^n \cup (1/\sqrt{p})B_2^n),
    $$
    and in particular, $(K_{p}({\cal E}))^\circ \sim B_\infty^n \cap \sqrt{p}B_2^n$.
\end{description}
Thus, the assertions of Theorem~\ref{thm:Gluskin} for $X=G$ and of Theorem~\ref{thm:LPRT} for $X={\cal E}$ can be formulated in a unified way: with high probability, it holds that
$$
{\rm absconv}(X_1,\ldots,X_N) \supset c_1 \bigl(K_{p}(X)\bigr)^\circ,
$$
for $p =c_2 \log(eN/n)$, where $c_1$ and $c_2$ are suitable constants.
Our main result shows that this phenomenon
holds under minimal assumptions on $X$, which we explain in the following. 

Let $\norm$ be a norm on $\R^n$ and set
$$
{\cal B} = {\cal B}_{\norm} =  \{x \in \R^n : \|x\| \leq 1\} \ \ \ {\rm and} \ \ \ {\cal S} = {\cal S}_{\norm} =\{x \in \R^n : \|x\|=1\}.
$$
The random vector $X$ is said to satisfy a small-ball condition with respect to the norm $\norm$ with constants $\gamma$ and $\delta$ if for every $t \in \R^n$,
\begin{equation} \label{eq:SB-intro}
\PROB( |\inr{X,t}| \geq \gamma \|t\| ) \geq \delta.
\end{equation}
Also, for some $r>0$, $X$ is said to satisfy an $L_r$ condition with respect to the norm $\norm$ and with constant $L$ if for every $t \in \R^n$,
\begin{equation} \label{eq:L-r-intro}
\left(\E |\inr{X,t}|^r\right)^{1/r} \leq L \|t\|.
\end{equation}
\begin{Assumption} \label{ass:main}
We assume that $X$ satisfies a small-ball condition with constants $\gamma>0$ and $\delta>0$, and an $L_r$ condition with constant $L$ for some $r>0$ with respect to the same norm $\norm$.
\end{Assumption}

Assumption \ref{ass:main} implies that the random vector $X$ is not degenerate: the small-ball condition \eqref{eq:SB-intro} means that marginals of $X$ do not have `too much' mass at $0$, and the $L_r$ condition \eqref{eq:L-r-intro} leads to some minimal uniform control on the tail decay of each marginal. Also, it is straightforward to verify from \eqref{eq:SB-intro} and \eqref{eq:L-r-intro} that
$$
\frac{1}{L} {\cal B} \subset B(L_r(X)) \subset \frac{1}{\gamma \delta^{1/r}} {\cal B}. 
$$

\medskip

Our answer to the first part of Question \ref{Qu:main} is that under  Assumption~\ref{ass:main}, a typical realization of ${\rm absconv}(X_1,\ldots,X_N)$ contains a constant 	multiple of $(K_{p}(X))^\circ$ for $p \sim \log(eN/n)$.
\begin{Theorem} \label{thm:main1}
Let $X$ be a symmetric random vector that satisfies Assumption \ref{ass:main} with respect to a norm $\norm$ and some $\delta, \gamma, r, L  > 0$. 
Let $0<\alpha<1$ and set $p=\alpha \log(eN/n)$ and assume that
 $N \geq c_0 n$ for a constant $c_0=c_0(\alpha,\delta,r,L/\gamma)$. Let $X_1,\hdots,X_N$ be independent copies of $X$ then with probability at least $1-2\exp(-c_1N^{1-\alpha} n^\alpha)$,
\begin{equation} \label{eq:main}
{\rm absconv}(X_1,\ldots,X_N) \supset \frac{1}{2}  \bigl(K_{p}(X)\bigr)^{\circ},
\end{equation}
where $c_1$ is an absolute constant.
\end{Theorem}

\begin{Remark}
Theorem~\ref{thm:main1} still holds -- even with the same constants and the same proof --
when ${\rm absconv}(X_1,\ldots,X_N)$
is replaced by the standard convex hull $\conv(X_1,\hdots,X_N)$, see also Remark~\ref{Rem:conv}.
\end{Remark}

Assumption~\ref{ass:main} is weaker than any of the assumptions in all previous results 
on the inner structure of ${\rm absconv}(X_1,\ldots,X_N)$. In particular, we allow heavy-tailed distributions and do not require independence of the entries of $X$. The freedom of choosing the norm $\norm$  makes the method very flexible.
Observe that \eqref{eq:main} does not depend on the specific choice of $\norm$,
but the constant $c_0=c_0(\alpha, \delta, r, L/\gamma)$ does. 
In fact, the constants $L$ and $\gamma$ may change when chaining the norm.
So the art consists in choosing a norm such that quotient $L/\gamma$, and hence, the constant $c_0$ 
become as small as possible.

As applications of Theorem~\ref{thm:main1} we show in Section~\ref{sec:K-bodies} how one can recover or improve the previous central results on the geometry of the random polytope ${\rm absconv}(X_1,\ldots,X_N)$ in this context. This is done by answering the second part of Question \ref{Qu:main}: we identify the floating bodied $K_p(X)$ in all those cases,
for example, when $X$ is the Gaussian vector (\cite{G}); when $X$ has \iid subgaussian centered coordinates (\cite{LPRT}); when $X$ is an isotropic, log-concave random vector (\cite{DGT}); and when $X$ has \iid centered coordinates that satisfy a small-ball condition (\cite{GLT}).

In addition, and thanks to the universality of Theorem~\ref{thm:main1}, one may establish various new outcomes that were previously completely out of reach like when $X$ is an unconditional random vector without necessarily independent entries, see Theorem \ref{thm:uncond}. The main applications we present  in this introduction are two results that we found to be particularly surprising: firstly, an answer to Question \ref{Qu:main} when $X$ has \iid  $q$-stable coordinates for $1 \leq q < 2$ (e.g., a Cauchy random vector); and secondly, an answer to a fundamental question on sparse recovery.
 
\subsection{Stable random vectors} \label{sec:intro-stable}
Consider standard $q$-stable random vectors for $1 \leq q <2 $ (a $2$-stable random vector is just a Gaussian), that is, vectors that have \iid standard $q$-stable random variables as coordinates. Recall that a random variable $\xi$ is standard $q$-stable if its characteristic function satisfies $\E[\exp(i t X)] = \exp(- |t|^q/2)$ for every $t \in \R$ (we consider only the symmetric case).
The following features of a standard $q$-stable random variable $\xi$ are of significance here:
\begin{description}
\item{$\bullet$} $\xi$ belongs to the weak-$L_q$ space; i.e., $\sup_{u>0} u^q \PROB(|\xi|>u) \leq C_q$, and for large values of $u$, $\PROB(|\xi|>u) \geq c_q/u^q$.
\item{$\bullet$} the stability property: if $\xi_1, \ldots,\xi_n$ are \iid copies of $\xi$ and $t \in \R^n$ then $\sum_{i=1}^n t_i \xi_i$ has the same distribution as $\|t\|_q \, \xi$.
\end{description}
	For a more comprehensive discussion on $q$-stable random variables see, e.g., \cite[Chapter 5]{LT}.
Note that for $q<2$, $\xi$ does not have a finite second moment,
which makes the analysis of the structure of the random polytope ${\rm absconv}(X_1,\hdots,X_N)$ more challenging.

The answer to Question \ref{Qu:main} for a $q$-stable random vector is as follows:
\begin{Theorem} \label{thm:stable}
Let $\xi$ be a standard, $q$-stable random variable for some $1 \leq q<2$. Let $\xi_1, \ldots,\xi_n$ be independent copies of $\xi$ and set $X=(\xi_i)_{i=1}^n$. Then for $0<\alpha<1$ and $N \geq c_0(\alpha,q)n$, with probability at least $1-2\exp(-c_1N^{1-\alpha}n^\alpha)$,
$$
{\rm absconv}(X_1,\ldots,X_N) \supset c_2(q) \left(\frac{N}{n}\right)^{\alpha/q} B_{q^\prime}^n
$$
where $1/q + 1/q^\prime = 1$.
\\
In particular, if $\xi$ is a  standard Cauchy random variable (corresponding to $q=1$) then
with probability at least $1-2\exp(-c_1N^{1-\alpha}n^\alpha)$
$$
{\rm absconv}(X_1,\ldots,X_N) \supset c_3 \left(\frac{N}{n}\right)^{\alpha} B_\infty^n.
$$
\end{Theorem}
Observe that a typical realization of ${\rm absconv}(X_1,\ldots,X_N)$ is much larger than, say, the typical realization of the random polytope generated by the Gaussian random vector. Indeed, the latter only contains $\sqrt{\log(eN/n)}B_2^n$, which is a much smaller set than $c({N}/{n})^{\alpha/q} B_{q^\prime}^n$. The intuitive reason behind this phenomenon is that for $q<2$, a $q$-stable random variable is more `heavy-tailed' than the Gaussian random variable: its tail decay is of the order of $u^{-q}$ rather than $\exp(-u^2/2)$ and that difference leads to the polynomial growth of the ``inner radius" of ${\rm absconv}(X_1,\ldots,X_N)$. At the same time, the difference in the canonical body contained in ${\rm absconv}(X_1,\ldots,X_N)$ is due to the natural metric associated with $X$: each marginal $\inr{t,X}$ is distributed as $\|t\|_q \, \xi $ rather than as $\|t\|_2 \, \xi$.

The proof of Theorem~\ref{thm:stable} is presented in Section~\ref{sec:stable}.

\subsection{Relation to Compressive Sensing} \label{sec:CS}

The second surprising outcome of Theorem~\ref{thm:main1} is related to a fundamental question in the area of \emph{compressive sensing}\footnote{For more information on compressive sensing we refer the reader to \cite{do06-2, CGLP, FoucartRauhut13}, and for more a detailed explanation on the connections between the geometry of random polytopes and sparse recovery, see \cite{DeVoreWojtaszczyk09,Wojtaszczyk10, CGLP, FoucartRauhut13,Foucart14,BA18}.}: can sparse signals be recovered efficiently when the given data consist of a few measurements that are noisy, \emph{but the `noise level' is not known}.

Suppose one would like to recover an unknown vector (signal) $x \in \R^N$ from an underdetermined set of a linear measurements,
i.e., from $y = A x \in \R^n$, where $A \in \R^{n \times N}$ with $n$ much smaller than $N$. While this is impossible
in general, the theory of compressive sensing studies when such recovery is possible by efficient methods
for \emph{($s$-)sparse vectors}, i.e., vectors in $\R^N$ that satisfy $\|x\|_0 = |\{\ell : x_\ell \neq 0\}| \leq s \ll n$.

One of the main achievement of compressive sensing was the discovery that a computationally efficient recovery procedure can be used to recover the signal. Indeed, if $x^\sharp$ is the solution of the $\ell_1$-minimization problem
\begin{equation}\label{l1min-equal}
\min_{z \in \R^N} \|z\|_1 \quad \mbox{ subject to } A z = y,
\end{equation}
then for a well-chosen $\sim s \log(eN/s)$ measurements, $x^\sharp$ coincides with the original $s$-sparse $x$. This upper estimate on the required number of measurements is optimal, and it is attained by a wide variety of random measurement ensembles---for example, if the measurements are $(\inr{G_i,x})_{i=1}^n$, i.e., $A$ is a draw of a random matrix with independent, mean-zero, variance one, Gaussian entries.

Naturally, to be of value in real-life applications, recovery should be possible in the presence of noise. The additional appeal of $\ell_1$-minimization is that it can be modified to perform well even if the given measurements $(\inr{Ax,e_i})_{i=1}^n$ are corrupted by noise, and if the signal $x$ is not necessarily sparse but only approximately sparse in some appropriate sense. Indeed, assume that the data one is given is $\bar{y} = Ax + w$ for $x \in \R^N$ and a vector $w \in \R^n$ of perturbations (noise) \emph{with a known noise level} $\|w\|_2 \leq \eta$. It is important to emphasize that unlike standard problems in statistics, here $w$ is an arbitrary vector, rather than a random draw according to some statistical law.

One can show that for a variety of random matrices, a sample size of $n \sim s \log(eN/s)$ suffices to ensure that
the minimizer $x^\sharp$ of the modified $\ell_1$-minimization problem
\begin{equation}\label{l1min-noise}
\min_{z \in \R^N} \|z\|_1 \quad \mbox{ subject to } \|A z - \bar{y} \|_2 \leq \eta
\end{equation}
satisfies
\begin{equation}\label{error-estimate}
\|x - x^\sharp\|_1 \lesssim \sigma_s(x)_1 + \eta \sqrt{\frac{s}{n}},
\end{equation}
where
$$
\sigma_s(x)_1 = \inf_{z:\|z\|_0\leq s} \|x-z\|_1
$$
is the best $\ell_1$ approximation error of $x$ by an $s$-sparse vector; again, this is the best estimate one can hope for.

Unfortunately, the $\ell_1$-minimization procedure of \eqref{l1min-noise} requires accurate information on the true noise level $\|w\|_2$, or at least a good upper estimate of it. However, in real world applications, this information is often not available.
Getting the noise level wrong renders the estimate \eqref{error-estimate} 	useless: if the employed value of $\eta$ is an underestimation of the true noise level
then the error bound \eqref{error-estimate} need not be valid.
On the other hand, if $\eta$ is chosen to be significantly larger
than the true noise level, the resulting error estimate \eqref{error-estimate} (involving the chosen $\eta$) is terribly loose.

As it happens, one can show that \emph{noise blind recovery}, in which the noise level is not known, is possible if the measurement matrix $A$ satisfies two conditions:
\begin{description}
\item{(1)} A version of the \emph{null space property} (NSP), see \eqref{def:NSP}. We refer the reader to \cite{CGLP, FoucartRauhut13} for a detailed exposition on the NSP.

    Identifying matrices that satisfy the null space property has been of considerable interest in recent years and many examples can be found, for example, in \cite{ALPT11b, CGLP, FoucartRauhut13, krmera14, ML17, DLR}. From our perspective, and somewhat inaccurately put, it is important to note that the NSP is (almost) a necessary condition for sparse recovery in noise-free problems. Therefore, to have any hope of successful recovery in noisy problems, the measurement matrix has to satisfy some version of the NSP.
\item{(2)} The second, and seemingly more restrictive condition is the so-called $\ell_1$ quotient property ~\cite{DeVoreWojtaszczyk09}. The matrix $A$ satisfies the $\ell_1$ quotient property with respect to the norm $\triple{ \ }$ if for every $w \in \R^n$ there exists a vector $v \in \R^N$ such that $A v = w$ and
\begin{equation}\label{l1:quotient:prop}
\|v \|_1 \leq \LL^{-1}  \triple{w}.
\end{equation}
\end{description}
It follows that if $A$ satisfies an appropriate null space property and the $\ell_1$-quotient property, then the solution $x^\sharp$ of \eqref{l1min-equal} for $y = Ax +w$ 
satisfies 
\begin{equation} \label{error-estimate-general}
\|x^\sharp - x\|_1 \lesssim \sigma_s(x)_1 + \triple{w};
\end{equation}
in other words, the noise-blind recovery error depends on `how far' $x$ is from being sparse and on the norm $\triple{w}$ of the noise vector. For the sake of completeness, an outline of the proof of \eqref{error-estimate-general} can be found in Appendix~\ref{Append:SparseRec}.

Theorem \ref{thm:main1} implies that contrary to prior belief, $(2)$ is not restrictive at all; in fact, it is almost universal. Indeed, let $\triple{ \cdot }_{p}$ be the norm whose unit ball is the polar body
$(K_{p}(X))^\circ$, i.e.,
\[
\triple{x}_{p} = \inf \{ t > 0 : x \in t (K_{p}(X))^\circ \}.
\]
Set $A =(X_1| \cdots | X_N)$ to be the random matrix whose columns are independent random draws
of the random vector $X$. Then the inclusion from Theorem~\ref{thm:main1} implies that for each vector $w \in \R^n$ there
exists a vector $v \in \R^N$ such that $A v = w$ and
\begin{equation}
\|v \|_1 \leq c_2^{-1}  \triple{w}_{p},
\end{equation}
which is precisely the $\ell_1$ quotient property with respect to the norm $\triple{\cdot}$.

\medskip

Thanks to the study of the floating bodies $K_{p}(X)$
presented in Section~\ref{sec:K-bodies}, the norm $\triple{ \ }_{p}$ can be identified in a variety of cases, and in some of which the appropriate null space property has already been established  -- leading the error bound \eqref{error-estimate-general}. These examples include some of the natural random ensembles that are used in sparse recovery, for example, when $X$ has \iid subgaussian or subexponential  coordinates \cite{ALPT11b,Foucart14}; when $X$ is an isotropic, log-concave random vector \cite{ALPT11b}; and when $X$ has independent coordinates that have $\log(N)$ finite moments \cite{ML17,DLR} (for example, when the coordinates are distributed according to the Student-$t$ distribution with $\sim \log N$ degrees of freedom).

Thanks to Theorem~\ref{thm:main1}, the $\ell_1$ quotient property can be established in those (and many other) cases, implying that noise-blind recovery is possible. To give a flavour of such a result, we present the example of the Student-t distribution in an appendix. More information and numerical experiments are given in \cite{KKR18}.

\section{Proof of the main result}

For the proof of Theorem~\ref{thm:main1}, we need some basic properties of the floating body
\[
K_{p}(X) = \left \{ t \in \R^n, \PROB(\langle X, t \rangle \ge 1) \le e^{-p} \right \}.
\]
Recall that a set $K$ is star-shaped around $0$ if for every $x \in K$ and any $0 \leq \lambda \leq 1$, $\lambda x \in K$.
\begin{Proposition} \label{prop:2.6}
Let $X$ be a symmetric random vector on $\R^n$. Then
\begin{description}
\item{$(1)$} The set $K_{p}(X)$ is star-shaped and symmetric around $0$. Moreover, for any $a > 0$,
\[
a K_p(X) = \left \{ t \in \R^n, \PROB(\langle X, t \rangle \ge a) \le e^{-p} \right \}.
\]
\item{$(2)$} Let $\norm$ be a norm on $\R^n$ and denote its unit ball by ${\cal B} =  {\cal B}_{\norm} = \{t \in \R^n : \|t\| \leq 1\}$. If $X$ satisfies the small-ball condition \eqref{eq:SB-intro} with respect to the norm $\norm$ with constants $\gamma$ and $\delta$, then for $p > \log(2/\delta)$,
\begin{equation} \label{eq:K-and-SB}
\gamma \, {\rm absconv}(K_{p}(X)) \subset {\cal B}.
\end{equation}
\item{$(3)$} If $X$ satisfies the $L_r$ condition \eqref{eq:L-r-intro} with respect to the norm $\norm$ with constant $L$, then
\begin{equation} \label{eq:K-and-ball-inside}
 {\cal B} \subset  L \, \exp(p/r) \, K_{p}(X).
\end{equation}
\end{description}
\end{Proposition}
\noindent{\bf Proof.}
The first  observation is straightforward.
To prove \eqref{eq:K-and-SB} observe that by convexity, it is enough to show that $ \gamma K_{p}(X) \subset {\cal B}$. But if $\|t\| \geq 1$ then the small-ball condition and the symmetry of $X$ imply that
$$
\PROB(\inr{X,t} \geq \gamma) \geq \PROB(\inr{X,t} \geq \gamma \|t\|) \geq \frac{\delta}{2} > \exp(-p),
$$
provided that $\delta > 2\exp(-p)$, as was assumed. Hence, $t \notin \gamma K_p(X)$.

As for \eqref{eq:K-and-ball-inside}, note that for $\|u\|\leq 1$, the $L_r$ condition yields that $\E |\inr{X,u}|^r \leq L^r$ and thus by Markov's inequality
$$
\PROB(\inr{X,u} \geq L \exp(p/r)) \leq \frac{\E |\inr{X,u}|^r}{L^r} \exp(-p) \leq \exp(-p),
$$
hence, $u \in L \exp(p/r) K_p(X)$.
\endproof

An outcome of Proposition \ref{prop:2.6} is that if $X$ satisfies Assumption \ref{ass:main} and
\begin{equation}\label{p:lowerbound}
p > \log(2/\delta),
\end{equation}
then $K_{p}(X)$ is a centrally symmetric subset of $\R^n$ that is star-shaped around $0$ and for which
\begin{equation} \label{eq:K-prop-1}
( 1 /  L) \, \exp(-p/r) \, {\cal B} \subset  K_{p}(X) \subset (1/\gamma) {\cal B}.
\end{equation}
Let ${\cal S}$ be the unit sphere of $(\R^n, \norm)$. For  $\theta \in {\cal S}$ set
$$
r(\theta) = \sup\{ \beta \geq 0 : \beta \theta \in  K_{p}(X)\}
$$
and note that by \eqref{eq:K-prop-1}, $(1 /  L) \, \exp(-p/r) \leq r(\theta) \leq 1/\gamma$. With a possible abuse of notation, put
\begin{equation}\label{def:partialK}
\partial K_p(X) = \{ r(\theta)\theta : \theta \in {\cal S} \}.
\end{equation}
Note that $\partial K_p(X)$ may not coincide with the topological boundary of $K_p(X)$ as $\theta \mapsto r(\theta)$ need not be continuous on ${\cal S}$ for general $X$.

\begin{Corollary} \label{cor:boundary-points}
For every $\theta \in {\cal S}$,
$$
\PROB( \inr{X,r(\theta) \theta} \geq 1) \geq \exp(-p).
$$
\end{Corollary}
\noindent{\bf Proof.}
It follows from the definition of $r(\theta)$ that for any $\rho > 1$, $\rho r(\theta) \theta \not \in  K_{p}(X)$, and thus,
$$
\PROB( \inr{X,\rho r(\theta) \theta} \geq 1) > \exp(-p).
$$
Taking the intersection of these events for any $\rho > 1$ gives the result.
\endproof

\medskip
The proof of Theorem~\ref{thm:main1} follows the path set in the (much simpler) proof of Theorem~1.5 from \cite{Men}. The goal is to show that if $X$ satisfies Assumption \ref{ass:main}, and
$$
N \geq c_0(\alpha,\delta, r, L/\gamma)\, n, \ \ \ p=\alpha \log(eN/n),
$$
then with probability at least
$$
1-2\exp(-c_1N^{1-\alpha}n^\alpha),
$$
one has that
\begin{equation}
\label{eq:inclusion}
\frac{1}{2} (K_{p}(X))^\circ \subset {\rm absconv}(X_1,\ldots,X_N).
\end{equation}
For a symmetric convex body $U$ with a nonempty interior,
define its support function $h_U$  by
$$
h_U(t) = \sup_{u \in U} \inr{u,t}, \ \ \ {\rm for \ all \ } t \in \R^n.
$$
The inclusion \eqref{eq:K-prop-1} ensures that $(K_{p}(X))^\circ$ has nonempty interior. Therefore, \eqref{eq:inclusion} is equivalent to
\[
\frac{1}{2} h_{(K_{p}(X))^\circ}(t) \le h_{{\rm absconv}(X_1,\ldots,X_N)}(t)
\]
for every $t \in \R^n$; and the negation of this event is that there exists $t \in \R^n$ such that
\begin{equation}
\label{eq:negation}
\frac{1}{2} \sup_{u \in (K_{p}(X))^\circ} \inr{u,t} > \sup_{v \in {\rm absconv}(X_1,\ldots,X_N)} \inr{v,t}.
\end{equation}
By homogeneity of \eqref{eq:negation} and since $\theta \mapsto r(\theta)$ is bounded away from $0$ on ${\cal S}$, it suffices to show that there is $t \in \partial K_p(X)$ for which  \eqref{eq:negation} holds.
Denote by $\Gamma : \R^n \to \R^N$  the random matrix whose rows are $X_1,\ldots,X_N$.  Observe that ${\rm absconv}(X_1,\ldots,X_N) = \Gamma^* B_1^N$, and therefore
\begin{equation}\label{Gamma-infty}
\sup_{u \in \Gamma^* B_1^N} \inr{u,t} = \sup_{x \in B_1^N} \inr{x,\Gamma t} = \|\Gamma t\|_\infty.
\end{equation}
Moreover, for $t \in \partial K_p(X)$, the definition of polarity gives $\sup_{u \in (K_{p}(X))^\circ} \inr{u,t} \le 1$.
Hence, for the proof of Theorem~\ref{thm:main1} it remains to show that
\begin{equation} \label{eq:what-to-prove}
\PROB \left( 
\inf_{t \in \partial K_{p}(X)} \|\Gamma t\|_\infty \leq 1/2 \right) \leq 2\exp(-c_1 \, N^{1-\alpha}n^\alpha).
\end{equation}

The proof of \eqref{eq:what-to-prove} is based on the small-ball method (see, for example, \cite{MendelsonLearning15}). First, fix any $t \in \partial K_{p}(X)$ and recall that by Corollary \ref{cor:boundary-points},
$$
\PROB(\inr{X,t} \geq 1) \geq \exp(-p).
$$
Therefore, by independence of the $X_i$ and Chernoff's inequality,
with probability at least
\begin{equation} \label{eq:in-proof-what-1}
1-\exp(-N \exp(-p)/8),
\end{equation}
it holds that
\begin{equation}\label{count:large-coord}
|\{i : \inr{X_i,t} \geq 1 \}| \geq \frac{N}{2}\exp(-p).
\end{equation}

Second, thanks to the high probability estimate \eqref{eq:in-proof-what-1}, it follows from the union bound
that if $T \subset \partial K_{p}(X)$ with
\begin{equation}\label{T-card}
|T| \leq \exp(N\exp(-p)/16),
\end{equation}
then
\begin{equation} \label{eq:uniform-est}
\inf_{t \in T} |\{i : \inr{X_i,t} \geq 1 \}| \geq \frac{N}{2}\exp(-p).
\end{equation}
with probability at least
\[
1-\exp(- N\exp(-p) / 16).
\]
The only restriction on the set $T \subset \partial K_{p}(X)$ is its cardinality. With this in mind, we will define  $T$ as a covering of $\partial K_{p}(X)$ with balls of  appropriate radius associated to the norm $\norm$.

Observe that by \eqref{eq:K-and-SB}, $\partial K_{p}(X) \subset (1/\gamma) {\cal B}$ provided that $p > \log(2/\delta)$.
By a standard volumetric estimate, see, e.g., \cite[Proposition~C.3]{FoucartRauhut13}, for every $\rho>0$ there exists a $\eta/\gamma$-cover of $\partial K_{p}(X)$
with respect to the norm $\norm$ of cardinality at most $(1 + 2/\eta)^n$. 
This $\eta/\gamma$-cover has the required cardinality \eqref{T-card} if 
\begin{equation} \label{eq:cond-on-eta}
\eta \geq 2 \left(\exp\left(\frac{N}{16n} \exp(-p)\right) - 1 \right)^{-1}.
\end{equation}
If
\begin{equation}\label{N:require}
N \geq 16 \ln(2) \exp(p) \, n
\end{equation}
then \eqref{eq:cond-on-eta} is satisfied for the choice
\begin{equation}\label{def:eta}
\eta = 4 \exp\left(-\frac{N}{16n} \exp(-p)\right).
\end{equation}
Denoting by ${\cal A}_1$ the event on which \eqref{eq:uniform-est} holds for $T$ that is a minimal $\eta/\gamma$-cover of $\partial K_p(X)$, it is evident that
$$
\PROB({\cal A}_1) \geq 1-\exp(- N\exp(-p)/16).
$$

Finally, for every $t \in \partial K_{p}(X)$ let $\pi t \in T$ be the nearest element to $t$ in the $(\eta / \gamma)$-cover with respect to the norm $\norm$. Consider the event ${\cal A}_2$ on which
\begin{equation} \label{eq:sup-process}
\sup_{t \in T} |\{i : |\inr{X_i,t-\pi t}| \geq 1/2\}| \leq \frac{3N}{8}\exp(-p).
\end{equation}
For each $t \in \partial K_p(X)$ consider the sets of indices
\[
I_1(t) := \{i : \inr{X_i, \pi t} \geq 1 \}, \quad I_2(t) := \{ i : |\inr{X_i,t-\pi t}| \geq 1/2\}.
\]
and observe that on the event ${\cal A}={\cal A}_1 \cap {\cal A}_2$,
\[
| I_1(t)| \geq \frac{N}{2}\exp(-p), \quad | I_2(t)| \leq \frac{3N}{8}\exp(-p).
\]
Clearly,
$$
| I_1(t)| + |I_2^c(t)| \geq \frac{N}{2}\exp(-p) + (N - \frac{3N}{8}\exp(-p)) = N + \frac{N}{8}\exp(-p)),
$$
and therefore
\[
|I_1(t) \cap I_2^c(t)| \geq  \frac{N}{8}\exp(-p).
\]
For each $t \in I(t) := I_1(t) \cap I_2^c(t)$ the triangle inequality gives 
$$
\inr{X_i,t} \geq \inr{X_i,\pi t} - |\inr{X_i,t-\pi t}| \geq \frac{1}{2}.
$$
In particular, on the event ${\cal A}$, it holds that
$\inf_{t \in \partial K_{p}(X)} \|\Gamma t\|_\infty \geq \frac{1}{2}$
and
$$
\PROB \left( \inf_{t \in \partial K_{p}(X)} \|\Gamma t\|_\infty \leq 1/2 \right) \leq \PROB\bigl( {\cal A}^c \bigr).
$$
Finally, let us show that $\PROB({\cal A}_2)$ is `large enough' for the right choice of $p$. To that end, observe that for every
$t \in \partial K_{p}(X)$, $\|t-\pi t\| \leq (\eta / \gamma)$, and therefore,
\begin{align*}
\sup_{t \in \partial K_{p}(X)} |\{i : |\inr{X_i,t-\pi t}| \geq 1/2\}| & \leq \sup_{u \in (\eta / \gamma) {\cal B}} |\{i : |\inr{X_i,u}| \geq 1/2\}|\\
& = \sup_{u \in (\eta / \gamma) {\cal B}} \sum_{i=1}^N \IND_{\{ |\inr{X_i,u}| \geq 1/2\}},
\end{align*}
which is the supremum of an empirical process indexed by the class of indicator functions
\begin{equation}\label{def:FInd}
\FF= \bigl\{ \IND_{\{ |\inr{\cdot,u}| \geq 1/2\}} : u \in (\eta / \gamma) {\cal B}\bigr\}.
\end{equation}
The wanted estimate on this supremum is based on an outcome of Talagrand's concentration inequality for bounded empirical processes, in the special case in which the indexing class is binary-valued and has
a finite Vapnik-Chervonenkis (VC) dimension (for a definition of the VC dimension, see, e.g., \cite{VC}).

Before stating this result, let us first recall the definition of VC dimension and a basic bound needed in our proof.

\begin{Definition} \label{def:VC}
Let $\FF$ be a class of $\{0,1\}$-valued functions on a space $\Omega$. The class shatters $\{x_1,\hdots,x_k\} \subset
\Omega$, if for every $I \subset \{1,\ldots,k \}$ there exists a function
$f_I \in \FF$ for which $f_I(x_i) =1$ if $i \in I$ and $f_I(x_i)=0$
if $ i \not \in I$. Let
\begin{equation*}
VC(\FF)=\sup \left\{ |A| \ : \  A \subset \Omega, \ A \
{\rm is \ shattered \ by} \ \FF \right\}.
\end{equation*}
\end{Definition}

\begin{Lemma}\label{VC-union} Let ${\cal D}$ be a set of subsets of $\Omega$ such that the set of indicator functions
$\FF = \{ \IND_{D} : D \subset {\cal D}\}$ satisfies $VC(\FF) = d$. If $\widetilde{\FF} = \{ \IND_{D \cup D'} : D, D' \in \mathcal{D}\}$ then $VC(\widetilde{\FF}) < 10 d$.
\end{Lemma}

\begin{proof} The statement is a special case of \cite[Lemma 3.2.3]{Blumer:1989tt}, which treats the case of the class of $k$ unions, i.e.,
${\cal D}^k = \{ D_1 \cup \cdots \cup D_k : D_1,\hdots,D_k \in {\cal D}\}$, and states that the VC dimension of the corresponding
class $\FF^k$ of indicator functions satisfies $VC(\FF^k) < 2d k \log_2(3k)$. For $k=2$ one has  $VC(\widetilde{\FF}) = c d$ with
$c = 4 \log_2(6) \approx 10.34$. The slightly better constant $10$ (or even $9.4$) follows from an inspection of the proof, which
shows that a strict upper bound for the VC dimension of $\widetilde{\FF}$ is any $m$ such that $(em/d)^{2d} < 2^m$. An explicit calculation
shows that $m=10d$ is a valid choice.
\end{proof}

Let us now state the outcome of Talagrand's concentration inequality when the indexing set of functions is a VC class (see \cite{ta94} and also \cite[Lemma 3.7]{ma03}).

\begin{Theorem} \label{thm:talagrand} 
Let $\FF$ be a class of $\{0,1\}$-valued functions for which $VC(\FF) \leq d$ and $\sup_{f \in \FF} \E f^2 \leq \sigma^2$.
Set
\begin{equation}\label{def:R}
R := 64 \frac{d}{N} \log\left(\frac{c}{\sigma^2}\right) + 8 \sigma \sqrt{\frac{d}{N} \log\left( \frac{c}{\sigma^2} \right)},
\end{equation}
where $c = 8e^2\sqrt{2} \approx 83.6$.
Then for any $x>0$,
\[
\PROB\left(\sup_{f \in \FF} \left|\frac{1}{N}\sum_{i=1}^N f(X_i) -\E f \right| \geq R + x\right) \leq
\exp\left(-N \frac{x^2/2}{\sigma^2 + 2R + x/3}\right).
\]
\end{Theorem}
For the sake of completeness, we provide a sketch of the argument in Appendix~\ref{Appendix:VC}.

\medskip

Let us return to the proof of Theorem~\ref{thm:main1} and  consider $\FF$ as defined in \eqref{def:FInd}. Each $f \in \FF$
is the indicator of a union of two half spaces in $\R^n$. By Radon's theorem,
the VC dimension of the class of indicators of half spaces in $\R^n$ is $n+1$, see e.g.~\cite[Theorem 3.4]{Mohri:2012tb}.
It follows from Lemma~\ref{VC-union} that $d:=VC(\FF) < 10(n+1)$.
Moreover, by the $L_r$  condition from Assumption \ref{ass:main} and Markov's inequality, for any $\theta  \in [0,1]$,
$$
\sup_{f \in \FF} \E f^2 = \sup_{u \in (\eta / \gamma) {\cal B}} \PROB\left( |\inr{X,u}| \geq \frac{1}{2}\right) \leq \min\left\{1,\left(\frac{2L\eta}{\gamma}\right)^r \right\}
\leq \min\left\{1,\left(\frac{2L\eta}{\gamma}\right)^{\theta r}\right\}.
$$
Hence, any $\sigma^2 \geq (2L\eta/ \gamma)^{ \theta r}$ is a valid choice in the context of Theorem~\ref{thm:talagrand}.
By our choice of $\eta$ in \eqref{def:eta}, this requirement is fulfilled for
\[
\sigma^2 = 
\left(\max\{1,8L/\gamma\} \right)^{\theta r} \exp\left( - \theta r \frac{N}{16n} \exp(-p) \right).
\]
With that choice of $\sigma$ and $p=\alpha \log(e N/n)$ the first term in the definition \eqref{def:R} of $R$
can be bounded as
\begin{align*}
T_1:= \frac{64d}{N}\log\left(\frac{c}{\sigma}\right) & < \frac{640(n+1)}{N}\left(\log(c)+\theta r\left(-
\log(\max\{1,8L/\gamma\}) + \frac{N}{16n} \exp(-p)\right) \right) \\
& \leq \frac{640\log(c)(n+1)}{N} + \frac{40(n+1)}{n} \theta r
\left(\frac{n}{eN}\right)^{\alpha}.
\end{align*}
Choosing $\theta = c_1 \min\{1,1/r\}$ with $c_1 = 1/(16\cdot 80)$ and assuming
$N \geq c_2 n$ for a suitable constant $c_2=c_2(\alpha)$, it is evident that  $\frac{640\log(c)(n+1)}{N} \leq  \left(\frac{n}{eN}\right)^\alpha/16$ and therefore,
\begin{equation}\label{T2:bound}
T_1 \leq \frac{1}{8}  \left(\frac{n}{eN}\right)^\alpha = \frac{1}{8} \exp(-p).
\end{equation}
Also, under the same assumptions, the second term in the definition \eqref{def:R} of $R$ can be estimated using \eqref{T2:bound} as
\begin{align*}
T_2 & := 8 \sigma \sqrt{\frac{d}{N} \log\left( \frac{c}{\sigma^2} \right)} \leq
8\sigma \sqrt{\frac{1}{8\cdot 64}\exp(-p)}\\
& \leq (\max\{1,8L/\gamma\})^{\theta r/2} \exp\left( - \theta r \frac{N}{32n} \exp(-p) \right) \sqrt{\frac{1}{8}\exp(-p)}\\
& = \frac{(\max\{1,8L/\gamma\})^{c_1\frac{\min\{1,r\}}{2}}}{\sqrt{8}}
\exp\left(- c_1\frac{\min\{1,r\}}{32e^\alpha}\left(\frac{N}{n}\right)^{1-\alpha} - \frac{\alpha}{2} \ln(eN/n)\right) \notag \\
&\leq \frac{1}{8} \exp(-\alpha \ln(eN/n)) = \frac{1}{8} \exp(-p), \label{T1:bound}
\end{align*}
provided that $N \geq c_3 n$ for some suitable $c_3 = c_3(\alpha,r,L/\gamma)$. Combining the two estimates, it follows that
\[
R = T_1 + T_2 \leq \frac{1}{4} \exp(-p).
\]
Moreover, with a similar argument we have that
\[
\sigma \leq \frac{1}{16}\exp(-p)
\]
provided that $N \geq c_4 n$ with $c_4 = c_4(\alpha,r,L/\gamma)$; furthermore,
\[
\sup_{f \in \FF} \E f \leq \sup_{f \in \FF} (\E f^2 )^{1/2} \leq \sigma \leq \frac{1}{16}\exp(-p).
\]
Now, recall that we assumed
\eqref{p:lowerbound}, i.e., that $p > \log(2/\delta)$, which by definition of $p$ is equivalent to $N > e(2/\delta)^{1/\alpha} n$. At the same time,
the requirement \eqref{N:require} is equivalent to $N \geq (16 \ln(2) e^\alpha)^{1/(1-\alpha)} n$.

Summarizing, all required conditions on $N$ are satisfied if
$N \geq c_0 n$ with
\[
c_0 = c_0(\alpha,r,L/\gamma,\delta) =
 \max\left\{c_2(\alpha),c_3(\alpha,r,L/\gamma), c_4(\alpha,r,L/\gamma), 3(2/\delta)^{1/\alpha}, (16 \ln(2) e^\alpha)^{1/(1-\alpha)}\right\}.
 \]
In this case, choosing $x = \exp(-p)/16$ in Theorem~\ref{thm:talagrand} and noting
that
\[
\sup_{f \in \FF} \left|\frac{1}{N} \sum_{i=1}^N f(X_i) \right| \leq \sup_{f \in \FF} \left|\frac{1}{N} \sum_{i=1}^N f(X_i) - \E f \right| + \sigma,
\]
it is evident that
\[
\sup_{u \in (\eta / \gamma) {\cal B}} \frac{1}{N} \sum_{i=1}^N \IND_{\left\{|\inr{X_i,z}\geq 1/2\right\}} \leq R + \sigma + x \leq \frac{3}{8} \exp(-p)
\]
outside an event whose probability is at most
\begin{align*}
\exp\left(-N \frac{x^2/2}{\sigma^2 + 2 R + x/3}\right)
&\leq \exp\left( - N \frac{\exp(-p)^2/(2\cdot 16^2)}{\exp(-p)^2/256 + \exp(-p)/2 + \exp(-p)/48}\right)  \\
&\leq  \exp\left(- c_6N \exp(-p)\right) = \exp(- c_1 N^{1-\alpha} n^{\alpha}),
\end{align*}
where $c_6 = 3/806$ and $c_1 = c_6/e \approx 0.0014$.
This completes the proof of \eqref{eq:sup-process} and, hence, of
Theorem~\ref{thm:main1}.
\endproof

\begin{Remark}\label{Rem:conv}
The proof only needs very little adaptation if one replaces ${\rm{absconv}}(X_1,\hdots,X_N)$ by the standard
convex hull $\conv(X_1,\hdots,X_N)$. In fact, $\conv(X_1,\hdots,X_N) = \Gamma^* \Delta_1^N$, where
$\Delta_1^N = \{x \in \R^N : x_i \geq 0, \sum_{i=1}^N x_i \leq 1\}$ is the standard simplex. Then 
$\|\Gamma t\|_\infty$ in \eqref{Gamma-infty} and \eqref{eq:what-to-prove} is replaced by $\max_{i=1,\hdots,N} (\Gamma t)_i = \max_{i=1,\hdots,N} \langle X_i,t\rangle$. Now, \eqref{count:large-coord} 
works without the absolute values around $\langle X_i,t \rangle$, anyway, so that the rest of the proof remains the same.
\end{Remark}

\section{The  floating bodies for various random vectors}
\label{sec:K-bodies}
Although Theorem~\ref{thm:main1} is (almost) universal, it is unrealistic to expect that the second part of Question \ref{Qu:main} can be addressed with a single result. Therefore, the identity of the sets $K_{p}(X)$ has to be studied on a case-by-case basis. Having said that, there are some general principles that can be used to identify, or at least approximate the sets $K_{p}(X)$. Firstly, as outlined in what follows, there are natural examples in which $K_{p}(X)$ can be identified directly---among them are the standard Gaussian vector $X=G$; the standard Rademacher vector $X={\cal E}$; and when $X$ is a $q$-stable random vector. Secondly, we show in
Section~\ref{sec:L-p} that if linear forms $\inr{X,t}$ have $p$-th moments and satisfy a weak regularity condition, then $K_{p}(X)$ is equivalent to $B(L_p(X))$. Perhaps, one could have actually expected a variant
of Theorem~\ref{thm:main1} with $K_{p}(X)$ replaced by  $B(L_p(X))$ in the first place,
but clearly $B(L_p(X))$ does not work in heavy-tailed situations, where it may be trivial if $p = \alpha \log(eN/n) > r$.
This observation, combined with Theorem~\ref{thm:main1} improves the main result from \cite{DGT} which studies random polytopes generated by isotropic, log-concave random vectors. Then, in Section \ref{sec:stochastic-dom}, we explain how stochastic domination can be translated to information on the structures of the floating bodies. That allows one to show that ${\rm absconv}(X_1,\ldots,X_N)$ contains large canonical sets for very general random vectors, even when $X$ does not necessarily have independent entries.

\subsection{Direct analysis of the floating body} \label{sec:directfloatingbody}
The first two natural examples one should consider are $X=G$, the standard Gaussian random vector and $X={\cal E}$, the standard Rademacher random vector. A direct computation shows that
$$
c_1 \frac{1}{\sqrt{p}} B_2^n \subset K_{p}(G)\subset c_2 \frac{1}{\sqrt{p}} B_2^n,
$$
and by \cite{Ms},
$$
c_1' {\rm conv}(B_1^n \cup (1/\sqrt{p})B_2^n) \subset K_{p}({\cal E}) \subset c_2'  {\rm conv}(B_1^n \cup (1/\sqrt{p})B_2^n),
$$
where $c_1$, $c_1'$, $c_2$ and $c_2'$ are absolute constants.
Therefore, in both cases, Theorem~\ref{thm:main1} implies that ${\rm absconv}(X_1,\ldots,X_N)$ contains a large canonical body. In particular, one recovers the estimates of Theorem~\ref{thm:Gluskin} and of Theorem~\ref{thm:LPRT} for the Rademacher random vector ${\cal E}$ stating that with high probability,
$$
{\rm absconv}(G_1,\ldots,G_N) \supset c_2 \sqrt{\alpha \log(eN/n)} B_2^n
$$
and
$$
{\rm absconv}({\cal E}_1,\ldots,{\cal E}_N) \supset c_2'  \bigl(B_\infty^n \cap \sqrt{ \alpha \log(eN/n)} B_2^n \bigr).
$$
We explain how Theorem~\ref{thm:LPRT} can be recovered from Theorem~\ref{thm:main1} in full generality in Section \ref{sec:stochastic-dom}.

\medskip
Another, more surprising example in which $K_{p}(X)$ can be computed directly consists in the case that $X$ is a standard $q$-stable random vector, a situation outlined in Theorem~\ref{thm:stable}.

\subsubsection{Proof of Theorem~\ref{thm:stable}} \label{sec:stable}
Recall that for $1 \le q < 2$, a  random variable $\xi$ is called standard $q$-stable  if its characteristic function satisfies $\E[\exp(i t X)] = \exp(- |t|^q/2)$ for every $t \in \R$ (we consider only the symmetric case).
The proof of Theorem~\ref{thm:stable} is based on several well known facts, see, e.g.,~\cite[Chapter 5]{LT}.
\begin{description}
\item{(F$1)$} If $\xi_1, \ldots, \xi_n$ are independent copies of a standard $q$-stable random variable $\xi$, and $X=(\xi_i)_{i=1}^n$, then for any $t \in \R^n$, $\inr{t,X}$ has the same distribution as $\xi \|t\|_q$.
\item{(F$2)$} While a standard $q$-stable random variable does not belong to $L_q$, it does belong to the weak-$L_q$ space $L_{q,\infty}$, i.e., $\sup_{u>0} u^q \PROB(|\xi| \geq u) \leq C_q$ for some constant $C_q > 0$.
\item{(F$3)$} The weak $L_q$ behaviour of $\xi$ is sharp: there exist constants $M_q, c_q > 0$ such that for any $u \geq M_q$, $\PROB(|\xi| \geq u) \geq c_q/u^q$.
\end{description}

From here on, let $\xi$ be a standard $q$-stable random variable for some $1 \leq q < 2$.
Let us first show that $X$ satisfies Assumption~\ref{ass:main}, though obviously, due to the stability property (F$1)$, \emph{not} with respect to the Euclidean norm, but rather with respect to $\norm_q$. By (F$2)$, $\xi$ has a bounded $L_r$ (quasi)-norm for any $0<r < q$. As a result, $X$ satisfies the $L_r$ condition \eqref{eq:L-r-intro} with respect to $\norm_q$ for $r=q/2$ and constant $L=L_q$. At the same time, e.g., by a Paley-Zygmund argument (see e.g. \cite[Chapter 3.3]{PenaGine12}), it is straightforward to verify that $X$ satisfies the small-ball condition \eqref{eq:SB-intro} with respect to $\norm_q$ for constants $\gamma=\gamma_q$ and $\delta=\delta_q$ that depend only on $q$.

Therefore, invoking Theorem~\ref{thm:main1}, a typical realization of ${\rm absconv}(X_1,\ldots,X_N)$ contains \linebreak $c  (K_{p}(X))^\circ$ for $p=\alpha \log(eN/n)$. It remains to identify the floating body $K_{p}(X)$.
To this end, observe that
\begin{equation} \label{eq:K-for-beta-stable}
K_{p}(X) \subset c_2(q)  \left(\frac{n}{N}\right)^{\alpha/q} B_{q}^n. 
\end{equation}
Indeed, let $t \in K_{p}(X)$. By (F$1)$, $\inr{X,t}$ has the same distribution as $\xi \|t\|_q$ and
$$
\PROB\left( \xi \geq \frac{1}{\|t\|_q}  \right)  = \PROB( \inr{X,t} \geq 1)
\leq \exp(-p) = \left(\frac{n}{eN}\right)^{\alpha}.
$$
Since $N/n$ is `large enough', it follows that for $M_q$ as in (F$3)$, $\|t\|_q \leq 1/M_q$; indeed, otherwise $\PROB(\xi \geq M_q) \leq (n/(eN))^\alpha$ which is impossible when $N/n$ is larger than a suitable constant. Now, by (F$3)$ ,
$$
c_q \left( \|t\|_{q}\right)^q \leq \PROB\left( \xi \geq \frac{1}{\|t\|_q}  \right) \leq \left( \frac{n}{eN}\right)^\alpha,
$$
implying that 
$$
\|t\|_q \leq c_2  \left(\frac{n}{eN}\right)^{\alpha/q},  
$$
where $c_2 = c_2(q) = c_q^{1/q}$. 
This establishes \eqref{eq:K-for-beta-stable} and completes the proof of Theorem~\ref{thm:stable} by taking the polar.
\endproof

\subsection{Floating bodies and the  unit ball of $L_p(X)$.}

\label{sec:L-p}
In order to get a better intuition on the role of the sets $K_{p}(X)$, let us consider a case in which $X$ is a `reasonably nice' random vector, in the sense that each $\inr{X,t}$ has sufficiently many moments and exhibits a weak kind of regularity. As we show next, the sets $K_{p}(X)$ are then equivalent to
$$
B(L_p(X))=\left\{t \in \R^n : (\E |\inr{X,t}|^p)^{1/p} \leq 1 \right\}, \quad p \geq  1.
$$
The polar body 
\begin{equation}\label{def:Zp}
Z_p(X) := B(L_p(X))^\circ 
\end{equation}
is called the $L_p$-centroid body of $X$. 
The fact that there is a connection between $K_{p}(X)$ and $B(L_p(X))$ is an immediate outcome of Markov's inequality:
$$
\PROB( \inr{X,t} \geq e \|\inr{X,t}\|_{L_p}) \leq \PROB( |\inr{X,t}|^{p} \geq e^{p} \|\inr{X,t}\|_{L_p}^{p}) \leq \exp(-p).
$$
Therefore, if $\|\inr{X,t}\|_{L_p} \leq 1/e$ then $t \in K_{p}(X)$, i.e.,
\begin{equation} \label{eq:L-p-1}
 \frac{1}{e} B\bigl(L_p(X)\bigr) \subset K_{p}(X).
\end{equation}
In order to prove a reverse inequality one requires an additional regularity condition on $X$.
\begin{Definition} \label{def:weak-regularity}
The random vector $X$ satisfies a regularity condition with constant $\LL$ if for every $q \geq 2$ and every $t \in \R^n$,
\begin{equation} \label{eq:weak-regularity}
\|\inr{t,X}\|_{L_{2q}} \leq \LL \|\inr{t,X}\|_{L_q}.
\end{equation}
\end{Definition}

\begin{Lemma} \label{lemma:L-p-2}
Let $X$ be a symmetric random vector for which \eqref{eq:weak-regularity} holds. Then, for every 
$p \geq c_2$, 
$$
K_{p}(X) \subset 2 B\bigl(L_{c_1p}(X)\bigr),
$$
where $c_1 = 1/(4\log(4\LL/3))$ and $c_2 = \max\{2 c_1, 2 \log(2)\}$. 
\end{Lemma}

\proof Fix $t \in \R^n$. By the symmetry of $X$,
$$
\PROB( \inr{X,t} \geq 1) = \frac{1}{2}\PROB( |\inr{X,t}| \geq 1),
$$
and invoking the Paley-Zygmund inequality (see, e.g.\cite[Chapter 3.3]{PenaGine12}) yields, for any $q \geq 2$,
\begin{align}
\label{PaleyZyg1}
\PROB\left( |\inr{X,t}| \geq \frac{1}{2} \|\inr{X,t}\|_{L_q} \right) & \geq
\left( (1-(1/2)^q) \frac{\|\inr{X,t}\|_{L_q}}{\|\inr{X,t}\|_{L_{2q}}}\right)^{2q} \geq \left(\frac{3}{4\LL}\right)^{2q} \\
&= \exp(-2q \log(4\LL/3)). \notag
\end{align}
Hence, if $q = c_1 p$ with $c_1 = c_1(\LL) = 1/(4 \log(4\LL/3))$ and $p \geq 2 \log(2)$ then
$$
\PROB\left( \inr{X,t} \geq \frac{1}{2} \|\inr{X,t}\|_{L_q} \right) \geq \frac{1}{2} \exp(- q/(2c_1))
\geq \exp(-p/2 -p/2) = \exp(-p).
$$
Hence, if $t \in K_{p}(X)$ then $\|\inr{X,t}\|_{L_q} \leq 2$, as claimed.
\endproof

\begin{Remark}
Note that in order to prove that $K_p(X) \subset 2B(L_p(X))$ for a fixed value of $p$ it suffices that $X$ satisfies that $\|\inr{X,t}\|_{L_{2q}} \leq \LL \|\inr{X,t}\|_{L_q}$ for $q=c^\prime p$.
\end{Remark}

\subsubsection*{Log-concave random vectors}

Let us give one generic example in which \eqref{eq:weak-regularity} holds and $K_{p}(X)$ is equivalent to $B(L_p(X))$. There are many other natural examples of random vectors that satisfy \eqref{eq:weak-regularity} (e.g., the Rade\-macher vector ${\cal E}$, thanks to Borell's hypercontractivity inequality \cite{Borell}), but since the focus of this note is on random polytopes generated by a heavy-tailed random vectors we will not pursue this direction further.

A random vector is log-concave if it has a density $f$ satisfying that for every $x,y$ in its support and any $0<\lambda<1$, $f(\lambda x +(1-\lambda)y) \geq f^\lambda(x) f^{1-\lambda}(y)$. 
The $L_p$-centroid bodies 
$Z_p(X)$ defined in \eqref{def:Zp} play a crucial role in the study of log-concave measures \cite{luzh97,Paouris06}.
For more information on log-concave random vectors we refer the reader to \cite{GeoIsoConvex14, GNT}.  

Let $X$ be a symmetric log-concave random vector that is non-degenerate, i.e., whose support is not contained
in a proper subspace of $\R^n$. 
It follows from Borell's inequality \cite{Borell} (see e.g. \cite[Proposition 5.16]{GNT}) that for every $t \in \R^n$ and $1 \leq p \leq q < \infty$,
\begin{equation}\label{Borell-ineq}
\|\inr{X,t}\|_{L_p} \leq \|\inr{X,t}\|_{L_q} \leq 12 \frac{q}{p} \|\inr{X,t}\|_{L_p}.
\end{equation}
Therefore, $X$ satisfies the weak regularity condition \eqref{eq:weak-regularity} with constant $\LL=24$, implying that $B(L_p(X)) \sim K_{p}(X)$. Further, by \eqref{PaleyZyg1} with $q=2$, $X$ satisfies 
a small-ball condition with respect to the norm $\|t\|_X := (\E|\langle X,t \rangle|^2)^{1/2} = \|\Sigma^{1/2} t\|_2$
with constants $\gamma = 1/2$ and $\delta = (1/32)^4$. Here  $\Sigma = \E XX^T$ is the covariance matrix of $X$, 
which is nonsingular by the non-degenerateness assumption on $X$ so that $\|\cdot\|_X$ is actually a norm.
Moreover, \eqref{Borell-ineq} also implies that  $X$ satisfies the $L_r$-condition for $r=4$ 
with respect to $\|t\|_X$ with $L=24$. 
Theorem~\ref{thm:main1} then leads to the following result. 
\begin{Theorem} \label{thm:poly-log-concave}
Let $X$ be a symmetric, non-degenerate, log-concave random vector.
Let $0<\alpha<1$, set $N \geq c_0(\alpha)n$ and put $p=\alpha \log(eN/n)$. Then, with probability at least $1-2\exp(-c_1 N^{1-\alpha} n^\alpha)$,
\begin{equation} \label{eq:in-thm-log-concave}
{\rm absconv}(X_1,\ldots,X_N) \supset c_2 Z_p(X)
\end{equation}
where $c_2$ is a universal constant.
\end{Theorem}

Theorem~\ref{thm:poly-log-concave} improves the main result from \cite{DGT}, which states that if $X$ is an isotropic (which means that its covariance matrix $\Sigma$ is the identity), log-concave random vector and $\Gamma$ is the random matrix whose rows are $X_1,\ldots,X_N$, then with probability at least
$1-2\exp(-c_1(\alpha) N^{1-\alpha} n^\alpha)-\PROB(\|\Gamma:\ell_2^n \to \ell_2^N\| \geq c\sqrt{N})$,
$$
{\rm absconv}(X_1,\ldots,X_N) \supset c_2(\alpha) Z_p(X).
$$

Thanks to the progress made in \cite{ALPT} in the study of random matrices with \iid isotropic log-concave rows, it is known that
$$
\PROB(\|\Gamma:\ell_2^n \to \ell_2^N\| \geq c\sqrt{N}) \leq \exp(-c^\prime \sqrt{n}).
$$
Therefore, the probability bound of the result in \cite{DGT} is 
 weaker than the one Theorem~\ref{thm:poly-log-concave}.

\subsection{Stochastic domination} 
\label{sec:stochastic-dom}

Up to this point, the examples focused on random vectors $X$ for which $K_{p}(X)$ can either be studied directly, or is equivalent to a natural convex body. One way of extending the scope of the analysis of the random polytopes ${\rm absconv}(X_1,\ldots,X_N)$ is by comparing the floating bodies $K_{p}(X)$ that are associated with different random vectors. As it happens, this comparison is simply a way of coding \emph{stochastic domination}.

\begin{Definition} \label{def:domination}
Let $X$ and $Y$ be centered random vectors in $\R^n$. The random vector $X$ dominates $Y$ with constants $\lambda_1$ and $\lambda_2$ if for every $t \in \R^n$ and every $u>0$,
$$
\PROB( \inr{X,t} \geq u) \geq \lambda_1 \PROB(\inr{Y,t} \geq \lambda_2 u).
$$
\end{Definition}
This means that  if $X$ dominates $Y$ with constants $\lambda_1$ and $\lambda_2$ then
\begin{equation} \label{eq:domination:inclusion}
K_{p}(X) \subset \lambda_2 K_{p^\prime}(Y)
\end{equation}
for $p^\prime=p-\log(1/\lambda_1)$.

\medskip

It is well known that this notion of domination is well-suited for the study of random vectors with \iid coordinates because it is preserved under tensorization:
\begin{Theorem} \label{thm:tensor} \cite{KW}
There are absolute constants $c_1$ and $c_2$ for which the following holds.
Let $x$ and $y$ be symmetric random variables and assume that for every $u>0$, $\PROB(x > u ) \geq \lambda_1 \PROB(y \geq \lambda_2 u)$. Let $x_1,\ldots,x_n$ be independent copies of $X$ and set $y_1,\ldots,y_n$ to be independent copies of $y$. Then $X=(x_i)_{i=1}^n$ dominates $Y=(y_i)_{i=1}^n$ with constants $c_1\lambda_1$ and $c_2\lambda_2$.
\end{Theorem}

Theorem~\ref{thm:tensor} leads to many structural results on ${\rm absconv}(X_1,\ldots,X_N)$ for vectors with \iid coordinates, by comparing $x$ to a canonical random variable like a Rademacher random variable (i.e., a symmetric, $\{-1,1\}$-valued random variable) or to the standard Gaussian random variable. 

Observe that if $x$ is a symmetric random variable that satisfies $\PROB(|x| \geq \gamma_0) \geq \delta_0$ then we have 
$$
\PROB(x \geq u) \geq \delta_0\PROB(\eps > u/\gamma_0),
$$
where $\eps$ is a Rademacher random variable. Hence, from Theorem~\ref{thm:tensor}, we get that 
if $x_1,\ldots, x_n$ are independent copies of $x$ and $X=(x_i)_{i=1}^n$, then $X$ dominates the Rademacher vector ${\cal E}$ with constants $\lambda_1$ and $\lambda_2$ that depend only on $\gamma_0$ and $\delta_0$. 
As a result, by \eqref{eq:domination:inclusion},
$$
K_{p}(X) \subset \lambda_2 K_{p^\prime}(\cal E),
$$
where $p^\prime = p -\log(1/\lambda_1)$.
Thanks to the characterization of $K_{p}({\cal E})$ and Theorem~\ref{thm:main1} one immediately recovers Theorem~\ref{thm:LPRT} as well as the main result from \cite{GLT}.
\begin{Theorem} \label{thm:domination1}
Let $x$ be a symmetric random variable that satisfies $\E x^2 =1$ and set $x_1,\hdots,x_n$ to be independent copies of $x$ and put $X=(x_i)_{i=1}^n$. If 
there are constants $\gamma$ and $\delta$ such that $\PROB(|x| \geq \gamma) \geq \delta$,
then for $N \geq c_0 n$, with probability at least $1-2\exp(-c_1N^{1-\alpha}n^\alpha)$,
$$
{\rm absconv}(X_1,\ldots,X_N) \supset c_2(B_\infty^n \cap \sqrt{ \alpha \log(eN/n)}B_2^n);
$$
here $c_0$  depends on $\alpha, \gamma$ and $\delta$,  $c_2$ depends on $\gamma$ and $\delta$, and $c_1$ is an absolute constant.
\end{Theorem}

The result can be pushed much further. 
The fact that $X$ has i.i.d.~coordinates can be relaxed to an unconditional assumption. Moreover, $X$ need not have a covariance, as in fact,
Assumption~\ref{ass:main} suffices to get the desired conclusion. 

\begin{Definition} \label{def:uncond}
A random vector $X=(x_i)_{i=1}^n$ is unconditional if for every $(\eps_i)_{i=1}^n \in \{-1,1\}^n$, $(x_i)_{i=1}^n$ has the same distribution as $(\eps_i x_i)_{i=1}^n$.
\end{Definition}

\begin{Theorem} \label{thm:uncond}
For every $0<\delta <1$ there is a constant $c=c(\delta)$ such that the following holds. 
Let $X$ be an unconditional random vector that satisfies the small-ball condition with constants $\gamma$ and 
$\delta$. Then, for any $p > c_0(\delta) = 4 \log(8/\delta) + \log(4)$,
$$
K_{p}(X) \subset \frac{c(\delta)}{\gamma} K_{p}({\cal E}).
$$
In particular, if $X$ satisfies Assumption \ref{ass:main} and $N \geq c_0(\alpha,\delta,r,L/\gamma)n$, then with probability at least $1-2\exp(-c_1 N^{1-\alpha}n^{\alpha})$,
$$
{\rm absconv}(X_1,\ldots,X_N) \supset \frac{1}{2} \bigl(K_{p}(X)\bigr)^{\circ} \supset c^\prime(\delta)\gamma \bigl(B_\infty^n \cap \sqrt{\alpha \log(eN/n)}B_2^n\bigr).
$$
\end{Theorem}

The proof of Theorem~\ref{thm:uncond} is based on contraction inequalities for the Rademacher random vector (see, e.g. \cite{LT}): if $|a_i| \leq |b_i|$ for $1 \leq i \leq n$ then for every $p \geq 1$,
\begin{equation}
\label{eq:averagecontraction}
\Bigl(\E \bigl| \sum_{i=1}^n \eps_i a_i \bigr|^p \Bigr)^{1/p} \leq \Bigl(\E \bigl| \sum_{i=1}^n \eps_i b_i \bigr|^p \Bigr)^{1/p},
\end{equation}
and for every $u >0$,
\begin{equation} \label{eq:contraction}
 \PROB\Bigl( \bigl|\sum_{i=1}^n \eps_i a_i \bigr| \geq u \Bigr) \leq  2\PROB\Bigl( \bigl|\sum_{i=1}^n \eps_i b_i \bigr| \geq u \Bigr).
\end{equation}
We also require Borell's hypercontractivity inequality \cite{Borell}: for every $t \in \R^n$ and $q > p >1$,
\begin{equation} \label{eq:hyper-Rad}
\|\inr{{\cal E},t}\|_{L_q} \leq \frac{q-1}{p-1} \|\inr{{\cal E},t}\|_{L_p}.
\end{equation}

\noindent {\bf Proof of Theorem~\ref{thm:uncond}.}
The second part of the theorem is an immediate outcome of the first part, Theorem~\ref{thm:main1}, and the fact that $X$ satisfies Assumption~\ref{ass:main}. To establish the first part, let us show that if $X = (x_1,\ldots, x_n)$ is an unconditional random vector and there are $\gamma, \delta > 0$ such that for every $1 \leq i \leq n$,
\begin{equation} \label{eq:smallballX}
\PROB( | x_i | \geq \gamma ) \geq \delta,
\end{equation}
then
$K_{p}(X) \subset \frac{c(\delta)}{\gamma} K_{p}({\cal E}).
$
Note that for this part of the theorem, $X$ does not need to satisfy the small-ball condition \ref{eq:SB-intro} for every direction,
but rather only for coordinate directions.

Let $t \in K_p(X)$. Since $X$ is unconditional and symmetric, it holds that
\begin{equation}
\label{eq:defKp}
 \frac{1}{2} \PROB_{X \otimes \eps} \Bigl( \bigl| \sum_{i=1}^n \eps_i |x_i| |t_i| \bigr| \geq 1 \Bigr)
= \PROB_X \Bigl( \sum_{i=1}^n x_i t_i \geq 1 \Bigr)
 \leq \exp(-p).
\end{equation}
Let $\phi:\R \to \R_+$ be the truncation at level $\gamma$, that is,
\begin{equation*}
\phi(z)=
\begin{cases}
|z| &\mbox{if } |z| \leq \gamma,
\\
\gamma & \mbox{if } |z| > \gamma,
\end{cases}
\end{equation*}
and set $Z_t = \sum_{i=1}^n \eps_i \phi(x_i) |t_i|$. Since $\phi(z) \le |z|$
the contraction principle \eqref{eq:contraction} yields, for every $(x_i)_{i=1}^n \in \R^n$,
\begin{align}
\PROB(|Z_t| \geq 1) &= \E_X \PROB_{\eps} (|Z_t| \geq 1)
\leq 2 \E_X \PROB_{\eps} \Bigl(\bigl|\sum_{i=1}^n \eps_i |x_i| |t_i|\bigr| \ge 1 \Bigr)
= 4 \PROB_X \Bigl( \sum_{i=1}^n x_i t_i \geq 1 \Bigr) \notag\\
& \leq 4 \exp(-p). \label{Prob:upper}
\end{align}
Observe that  for every $1 \leq i \leq n$,
$$
\E_X \phi(x_i) \ge \gamma \PROB( |x_i| \ge \gamma) \geq \gamma \delta,
$$
where the last inequality follows from the small ball assumption \eqref{eq:smallballX}.
This observation implies that for any $q > 1$,

\begin{align}
(\E |Z_t|^{2q} )^{1/2q} &  \leq \gamma \Bigl(\E_\eps \bigl(\sum_{i=1}^n \eps_i  t_i\bigr)^{2q} \Bigr)^{1/2q}
\leq \frac{2q-1}{q-1} \gamma \Bigl(\E_\eps \bigr|\sum_{i=1}^n \eps_i  t_i\bigr|^q \Bigr)^{1/q} \notag\\
&\leq  \frac{2q-1}{q-1} \delta^{-1} \Bigl(\E_\eps \bigl|\sum_{i=1}^n \eps_i \E_X \phi(x_i) |t_i| \bigr|^q \Bigr)^{1/q}
\leq \frac{2q-1}{q-1} \delta^{-1} (\E |Z_t|^q)^{1/q}. \label{Zt:est}
\end{align}
Here, the first inequality used that $\phi(z) \leq \gamma$ as well as the contraction principle \eqref{eq:averagecontraction}, the second inequality is based on the hypercontractivity inequality for the Rademacher vector \eqref{eq:hyper-Rad} and the last inequality follows from Jensen's inequality.
Therefore, by the Paley-Zygmund inequality (as in, e.g., \cite[Chapter 3.3]{PenaGine12}), we have that  \begin{align}
\PROB\left(|Z_t| \ge (\E |Z_t|^q/2 )^{1/q} \right)
&\geq\left(\frac{1}{2} \frac{ (\E |Z_t|^q )^{1/q} }{(\E |Z_t|^{2q} )^{1/2q} } \right)^{2q}
\geq \left( \frac{\delta(q-1)}{2(2q-1)}\right)^{2q}\notag \\
& = \exp\left(-2q \log\left(\frac{4q-2}{(q-1)\delta}\right)\right). \label{Prob:lower}
\end{align}
If  $q$ is such that
\begin{equation}\label{cond:qp}
2q \log\left(\frac{4q-2}{(q-1)\delta}\right) < p - \log(4),
\end{equation}
then it follows that $\E|Z_t|^q < 2$ because otherwise \eqref{Prob:lower} would be in contraction to \eqref{Prob:upper}.
Before elaborating on the implication of $\E|Z_t|^q < 2$, let us discuss the particular choice
\[
q = \frac{p - \log(4)}{2 \log(8/\delta)}.
\]
Since $p > 4 \log(8/\delta) + \log(4)$ by assumption, it follows that $q > 2$ and $q-1 > q/2$ so that
$(4q-2)/(q-1) < 8-4/q < 8$ and
\[
2q \log\left(\frac{4q-2}{(q-1)\delta}\right) < 2q \log(8/\delta) = p- \log(4),
\]
so that \eqref{cond:qp} is satisfied.
Note that since $q > 2$ and $p > 4 \log(8/\delta) + \log(4)$,
\begin{align*}
C_{p,q} &:= \frac{p-1}{q-1} < \frac{p-1}{q/2} = \frac{p-1}{p-\log(4)} 4 \log(8/\delta)  = \left(1+\frac{\log(4)-1}{p - \log(4)}\right) 4 \log(8/\delta)\\
&< 4\log(8/\delta) + \log(4) - 1 =: C_\delta.
\end{align*}
By hypercontractivity combined with \eqref{Zt:est} (starting with the term after the second inequality in the first line)
and the observation that $(\E|Z_t|^q)^{1/q} < 2^{1/q} < \sqrt{2}$, we obtain
\begin{align*}
\left( \E_\eps \left| \sum_{i=1}^n \varepsilon_i t_i \right|^p \right)^{1/p}
\leq \frac{p-1}{q-1}
\left( \E_\eps \left| \sum_{i=1}^n \varepsilon_i t_i \right|^q \right)^{1/q}
\leq C_{p,q} \frac{1}{\gamma \delta} (\E |Z_t|^q)^{1/q} <  C_\delta \frac{\sqrt{2}}{\gamma \delta} =: C(\delta,\gamma).
\end{align*}
Markov's inequality gives
\[
\PROB_\eps \left(\sum_{i=1}^n \eps_i  t_i \geq e C(\delta,\gamma)\right) \leq \exp(-p).
\]
Hence, for
$$
c(\delta) = \frac{\sqrt{2}e(4 \log(8/\delta) + \log(4/e))}{\delta}
$$
it holds that $c(\delta)/\gamma = e C(\delta,\gamma)$ and
$
K_{p} (X) 
\subset \frac{c(\delta)}{\gamma} K_{p} ({\cal E})
$
as claimed.
\endproof

\appendix

\section{Concentration inequality for VC classes of functions}
\label{Appendix:VC}
We prove Theorem~\ref{thm:talagrand} in this section, basically following \cite{ma03}
but with a simplification (avoiding the use of \cite[Lemma 3.6]{ma03} due to Talagrand \cite{ta94}). 
The main tool is the following version of Talagrand's concentration
inequality \cite{ta96-2} due to Bousquet \cite{bo02-1},
see also \cite[Theorem~8.42]{FoucartRauhut13}, which features explicit and small constants.

\begin{Theorem}  Let $\GG$ be a set of functions $g: \R^n \to \R$.
Let $X_1,\hdots,X_N$ be independent random vectors in $\R^n$ such that $\E g(X_\ell) = 0$ and
$|g(X_\ell)| \leq K$ almost surely for all $\ell = 1,\hdots,N$ and for all $g \in \FF$ for some constant $K > 0$.
Introduce
\begin{equation}\label{def:empirical}
Z = \sup_{g \in \GG} \left|\sum_{\ell=1}^N g(X_\ell)\right|.
\end{equation}
Let $\sigma^2_\ell>0$ such
that $\E \left[g(X_\ell)^2\right] \leq \sigma_\ell^2$ for all $g \in \GG$ and $\ell = 1,\hdots,N$.
Then, for all $t > 0$,
\begin{equation}\label{empirical:deviation}
\PROB(Z \geq \E Z + t) \leq \exp\left( - \frac{t^2/2}{\sigma^2_\GG + 2 K \E Z + tK/3} \right),
\end{equation}
where $\sigma^2_\GG = \sum_{\ell=1}^N \sigma_\ell^2$.
\end{Theorem}

In the situation of Theorem~\ref{thm:talagrand},
we consider $\GG = \{ g = f - \E[f(X)] : f \in \FF\}$, so that $\E g(X) = 0$ and $|g(X)| \leq 1 =: K$ almost surely
for all $g \in \GG$. Moreover, $\sigma_\ell^2 \leq \sigma^2 = \sup_{f \in \FF} \E[f(X_\ell)^2]$ so that $\sigma_\GG^2 \leq N\sigma^2$. It remains to estimate $\E Z$.

Symmetrization, see e.g.\ \cite[Lemma 6.3]{LT}, and Dudley's inequality in the form of \cite[Theorem 8.23]{FoucartRauhut13} yield, for a Rademacher sequence $\varepsilon_1,\hdots,\varepsilon_N$ independent of $X_1,\hdots,X_N$,
\begin{align*}
\E Z &= \E \sup_{f \in \FF} \left|\sum_{j=1}^N f(X_\ell) - \E[f(X_\ell)] \right|
\leq 2 \E \sup_{f \in \FF} \left| \sum_{j=1}^N \varepsilon_j f(X_\ell)\right| \\
&
\leq 8 \sqrt{2N} \E_X \int_0^{\Delta_X(\FF)/2}
\sqrt{ \log(2 \mathcal{N}(\FF, d_{X,2},u))} du,
\end{align*}
where the metric $d_{X,2}$ is given as
\[
d_{X,2}(f,g) = \left( \frac{1}{N} \sum_{j=1}^N (f(X_j) - g(X_j))^2 \right)^{1/2}
= \left(\E_\varepsilon \left( \frac{1}{\sqrt{N}} \sum_{j=1}^N \varepsilon_j (f(X_j) - g(X_j))\right)^2 \right)^{1/2}
\]
and $\mathcal{N}(\FF, d_{X,2},u)$ denote the  covering numbers of $\FF$, i.e., the minimal number of balls of radius $u$ in the metric $d_{X,2}$ required to cover $\FF$ and
\[
\Delta_X(\FF) := \left(\sup_{f \in \FF} \frac{1}{N}\sum_{j=1}^N f(X_j)^2 \right)^{1/2} =
\Big( \underbrace{\sup_{f \in \FF} \frac{1}{N}\sum_{j=1}^N f(X_j)}_{=:Y} \Big)^{1/2},
\]
where we have used that $f$ takes only values in $\{0,1\}$ in the equality.
It follows from Haussler's theorem \cite{Haussler1995} and the fact that
the functions in $\FF$ are $\{0,1\}$-valued (so that $\|f-g\|_{L_2(\mu)}^2 = \|f-g\|_{L_1(\mu)}$ for any probability measure $\mu$) that the covering numbers can be estimated via the VC-dimension $d$ as
\[
\mathcal{N}(\FF, d_{X,2},t) \leq e (d+1) (2e)^d t^{-2d}.
\]
Plugging this into our estimate of $\E Z$ above and noting that $(2e(d+1))^{1/2d}$ takes the maximum for $d=1$, so that $(2e(d+1))^{1/2d} \leq 2 \sqrt{e}$ for all $d \geq 1$, gives
\[
\E Z \leq 8 \sqrt{2N} \E \int_0^{\sqrt{Y}/2} \sqrt{\log(2e(d+1)(2e/u)^{2d})} du
\leq 16 \sqrt{Nd} \E \int_0^{\sqrt{Y}/2} \sqrt{\log(2 e \sqrt{2} /u)} du
\]
We use the Cauchy-Schwarz inequality to estimate the integral
\[
\int_0^\alpha \sqrt{\log(\gamma/u)}du \leq \sqrt{\int_0^\alpha 1 du} \sqrt{\int_0^\alpha \log(\gamma/u)} du = \sqrt{\alpha} \sqrt{\gamma \int_{\gamma/\alpha}^\infty \log(t) t^{-2} dt}
= \alpha \sqrt{\log(e\gamma/\alpha)}.
\]
Setting $\alpha = \sqrt{Y}/2$ and $\gamma = 2 e \sqrt{2}$, noting that
$t \mapsto \sqrt{t \log(2e\sqrt{2}/t)}$ is concave and applying Jensen's inequality gives
\[
\E Z \leq 16 \sqrt{Nd} \E \sqrt{\frac{Y}{4} \log\left(\frac{2e^2 \sqrt{2}}{Y/4}\right)}
\leq 16 \sqrt{Nd} \sqrt{ \frac{\E Y}{4} \log\left(\frac{2e^2\sqrt{2}}{\E Y /4}\right)}.
\]
Now observe that by the triangle inequality and since each $f$ takes values in $\{0,1\}$,
\[
Y \leq Z/N + \sup_{f \in \FF}  \E f(X) = Z/N + \sup_{f \in \FF} \E f^2(X) \leq Z/N + \sigma^2.
\]
Since $t \mapsto \sqrt{t \log(2e^{2}\sqrt{2}/t)}$ is increasing, this yields
\[
\E Z \leq 8 \sqrt{Nd} \sqrt{\frac{\E Z + N\sigma^2}{N}\log\left(\frac{8e^2\sqrt{2}}{\E Z/N + \sigma^2}\right)}
\leq 8 \sqrt{d} \sqrt{(\E Z + N \sigma^2)\log\left(\frac{8e^2\sqrt{2}}{\sigma^2}\right)}.
\]
Setting $Q := 8 \sqrt{d\log\left(\frac{8e^2\sqrt{2}}{\sigma^2}\right)}$ and squaring leads to the inequality $(\E Z)^2 \leq Q^2(\E Z + N\sigma^2)$ so that
\begin{align*}
\E Z &\leq Q^2/2 + \sqrt{Q^2 N \sigma^2 + Q^4/4} \leq Q^2 + Q \sqrt{N} \sigma \\
& = 64 d \log\left(\frac{8e^2\sqrt{2}}{\sigma^2}\right) + 8 \sigma \sqrt{Nd  \log\left(\frac{8e^2\sqrt{2}}{\sigma^2}\right)} = NR.
\end{align*}
It follows from \eqref{empirical:deviation} that
\[
\PROB \left( \sup_{f \in \FF} \big|\sum_{j=1}^N (f(X_j) - \E f(X_j))\big| \geq NR + t\right) \leq
\exp\left(-\frac{t^2/2}{N \sigma^2 + 2 NR + t/3}\right),
\]
which is equivalent to the statement of Theorem~\eqref{thm:talagrand}.

\section{Sparse recovery}
\label{Append:SparseRec}

We begin this section with an outline of the proof of how the $\ell_1$-quotient property leads to \eqref{error-estimate-general}. The null space property of $A$ of order $s$ with constant $\rho < 1$
requiring that 
\begin{equation}\label{def:NSP}
\sum_{j \in S} |v_j|  \leq \rho \sum_{j \in S^c} |v_j| \quad \mbox{ for all } v \in \ker A \setminus \{0\}
\mbox{ and all } S \subset \{1,\hdots,N\} \mbox{ with } \#S = s,
\end{equation} 
implies by  \cite[Theorem 4.12]{FoucartRauhut13} that the solution $x^\sharp$ of equality constrained
$\ell_1$-minimization \eqref{l1min-equal} with $y = Ax$ satisfies
\begin{equation}\label{err:bound:NSP}
\|x-x^\sharp\|_1 \leq \frac{2(1+\rho)}{1-\rho} \sigma_s(x)_1.
\end{equation}
If $y = Ax + w$, then the $\ell_1$-quotient property yields the existence of $v \in \R^n$
satisfying \eqref{l1:quotient:prop}, so that we can write $y = A(x+v)$. The error bound \eqref{err:bound:NSP} then leads to
\begin{align*}
\|x^\sharp - x\|_1& \leq \frac{2(1+\rho)}{1-\rho} \inf_{z:\|z\|_0 \leq s} \|x+v-z\|_1
\leq  \frac{2(1+\rho)}{1-\rho} \left(\inf_{z:\|z\|_0 \leq s} \|x-z\|_1 + \|v\|_1\right) \notag\\
& \leq \frac{2(1+\rho)}{1-\rho} \sigma_s(x)_1 +  \frac{2(1+\rho)}{c_2(1-\rho)} \triple{w}_{p} ,
\end{align*}
which is \eqref{error-estimate-general}.

\medskip

Next, let us turn to the example of noise-blind recovery when the measurement matrix has \iid columns, selected according to the random vector $X$, which has \iid coordinates, distributed according to the ($L_2$-normalized) Student-$t$ entries with $d=2\log N$ degrees of freedoms. In particular, the first moments of each coordinate are equivalent to that of a Gaussian random variable: for any $q \leq \log N$, $c_1 \|g\|_{L_q} \leq \|\xi\|_{L_q} \leq c_2 \|g\|_{L_q}$.
This example is particularly interesting
because it was recently shown (see, e.g., \cite{ML17} and \cite[Example 9]{DLR}) that the corresponding
random matrix satisfies the null space property \eqref{def:NSP} of order $s$ with high probability as long as 
$s \sim n/\log(eN/n)$. In addition, numerical tests in \cite{DLR} show that this random matrix behaves
precisely like a Gaussian random matrix in practical sparse recovery problems. 
However, the $\ell_1$-quotient
property of a Student-$t$ matrix was previously open.

It is straightforward to verify that for any $q \leq \log N$ and every $w \in \R^n$, $\|\inr{X,w}\|_{L_q} \sim \|\inr{G,w}\|_{L_q}$. Moreover, setting $p = \alpha \log(eN/n)$, the results in Section \ref{sec:L-p} imply that
$$
K_p(X) \sim B(L_p(X)) \sim B(L_p(G)) \sim \sqrt{\log(eN/n)} B_2^n;
$$
therefore,
$$
\triple{w}_{p} \sim_\alpha \|w\|_2 \sqrt{\log(eN/n)}.
$$
The general error estimate \eqref{error-estimate-general} and Theorem~\ref{thm:main1} together with
$s \sim n/\log(eN/n)$  lead to
\begin{equation}\label{error-estimate-quot}
\|x^\sharp - x\|_1 \lesssim \sigma_s(x)_1 + \sqrt{\log(eN/n)} \|w\|_2 \sim \sigma_s(x)_1 + \|w\|_2\sqrt{\frac{s}{n}} .
\end{equation}
Note that \eqref{error-estimate-quot} yields the same error estimate as \eqref{error-estimate} (up to absolute constants), but while \eqref{error-estimate} requires an a priori threshold for the noise level, \eqref{error-estimate-quot} does not, and the error depends on the true noise level $\|w\|_2$ rather than a potentially pessimistic upper bound. We refer to \cite{KKR18} for more results in this direction and corresponding numerical experiments.

\subsection*{Acknowledgements}

HR would like to thank the Isaac Newton Institute for Mathematical Science for support and hospitality during the program Approximation, Sampling and Compression in Data Science when work on this paper was undertaken. This work was supported by EPSRC Grant Number EP/R014604/1.

\noindent
OG thanks the funding of the Fondation Simone et Cino Del Duca for  the project "Ph\'enom\`enes en grande dimension".

\noindent
FK was supported by the German Science Foundations in the context of an Emmy Noether Junior Research Group (KR 4512/1-1)

\medskip

\noindent
{ \bf AMS 2010 Classification:}
{
primary: 52A22, 46B06, 60B20, 65K10
secondary: 52A23, 46B09, 15B52.
}

\noindent
{ \bf Keywords: }
{
Random polytopes, random matrices, heavy tails, 
small ball probability,
compressed sensing, $\ell_1$-quotient property.}

\bibliographystyle{abbrv}
\bibliography{random-poly}

\end{document}